\documentclass{article}

\usepackage{amsmath,amsfonts,amssymb,amsthm}
\usepackage{showkeys}

\usepackage{geometry}
\newtheorem{theorem}{Theorem}
\newtheorem{lemma}[theorem]{Lemma}

\newtheorem{remark}{Remark}
\newtheorem{definition}{Definition}
\newtheorem{proposition}{Proposition}

\newcommand{\sectionnew}[1]{
\section{#1}\setcounter{equation}{0}
\setcounter{theorem}{0}}
\newcommand{\R}{\mathbb{R}}
\newcommand{\Z}{\mathbb{Z}}
\newcommand{\N}{\mathbb{N}}
\newcommand{\T}{\mathbb{T}}
\newcommand{\ep}{\epsilon}

\newcommand{\para}{{_{||}}}

\newcommand{\xg}{{x_{g}}}
\newcommand{\vg}{{v_g}}

\newcommand{\alphab}{\bar{\alpha}}
\newcommand{\Frho}{\hat{\rho}}
\newcommand{\tf}{{\tilde f}}
\newcommand{\tm}{{\tilde{m}}}

\DeclareMathOperator{\diver}{div}
\DeclareMathOperator{\signe}{sign}

\newcommand{\cqfd}
{%
\mbox{}%
\nolinebreak%
\hfill%
\rule{2mm}{2mm}%
\newline
\newline
}

\newcommand{\au}{{\vec{u}}}

\title{WELL-POSEDNESS OF A DIFFUSIVE GYROKINETIC MODEL.}
\author{MAXIME HAURAY and ANNE NOURI\\
\\LATP, Aix-Marseille University, France}

\date{}

\begin{document}

\maketitle

{\noindent \bf Abstract.}\hspace{0.1in}
We study a finite Larmor radius model used to describe the ions distributions in the core of a tokamak plasma, that consist in a  gyro-kinetic transport equation,  coupled with an electro-neutrality equation. Since the last equation do not provide enough regularity on the electric potential, we introduce a simple linear collision operator adapted to the finite Larmor radius approximation. Next we study the two-dimensional dynamics in the direction perpendicular to the magnetic field and prove thanks to the smoothing effects of the collisions and of the gyro-average the global existence of solutions, as well as short time uniqueness and stability.\\


\sectionnew{Introduction.}

  The model studied in that article describes the density of ions in the core of a tokamak plasma. 
In such highly magnetized plasma, the charged particles have a very fast motion of gyration around the magnetic lines, called the Larmor gyration. A good approximation is then to consider that the particles are uniformly distributed on gyro-circle, parametrized by their gyro-center, and Larmor radius $r_L$ (that is proportionnal to the speed of rotation $u$, and in our article, we will forget the physical constant and write $r_L = u$). The models obtained in that new variables are kinetic in the direction parallel to the magnetic field lines, and fluids (precisely a superposition of fluid models) in the perpendicular direction. For rigorous derivation of such models and more complete discussion on its validity, we refer to \cite{FreSon} and our previous work \cite{GheHauNou09}, in which the derivation is perform from a Vlasov equation in the limit of large magnetic field. 

  Such gyro-kinetic models  are usually closed by an electro-neutrality equation, that as usual provide very few regularity for the eletric field, so that the well-posedness of gyro-kinetic models is, at least at our knowledge unknown. In this article, we add a "gyro-averaged" collision operator to the model and study the dynamics in the directions perpendicular to the field only. 

Let us now describe our precise model.  The ion distribution
function $f(t,x,u)$ in gyro-coordinates depends on the time $t$, the gyro-center position $x \in
\T^2$ and the velocity of the fast Larmor rotation $u \in \R^+$ (which is also
proportional to the Larmor radius). The electric potential $\Phi$ depends only
on $(t,x)$. They satisfy the following system of equation on $\Omega = \T^2
\times \R^+$ 
\begin{align} 
 \frac{\partial f}{\partial t}   + (J^0_u \nabla_x \Phi)^\perp \cdot \nabla_x
f = \beta
u \partial_u f + 2\beta  f + \nu \left( \Delta_x f  + \frac{1}{u} \partial_u ( u
\partial_u f) \right)  \label{eq:gyroFP2D}   \\
(\Phi   -  \Phi \ast _x H_T  )(t,x)=  T \left( \rho(t,x)  -1 \right)
\label{eq:elecneutr} \\
 \rho(t,x) = \int (J^0_w f(t,x,w)2\pi wdw) \label{eq:defrho} \\
f(0,x,v) = f_i(x,v) \qquad \forall (x,u) \in \Omega \label{eq:indata}
\end{align}
where $\beta$ and $\nu$ are two positive constant,  $\rho$ is the density in physical space,
$T$ is the ion temperature,
\begin{equation} \label{eq:J0}
J^0_{u} h(x_g) = \frac{1}{2\pi} \int_0^{2\pi} h(x_g +u e^{i\varphi_c}) \,d\varphi_c \, ,
\end{equation}
is the well known zero-order Bessel operator \cite{Watson} and
\begin{equation} \label{eq:dfH}
H_T(x)= \frac{e^{-\frac{| x| ^2}{4T}}}{2\pi ^{\frac{3}{2}}\sqrt{T} | x| } \,.
\end{equation}
We also used the notation $b^\perp =(-b_2,b_1)$, for any vector
$b=(b_1,b_2)$ of $\R^2$.  

\medskip
That model without the Fokker-Planck
operator ($\nu=\beta=0$) was studied in
our previous work \cite{GheHauNou09} - to which we refer for an heuristical
derivation of the electro-neutrality equation \eqref{eq:elecneutr} - and is
used by physicists for simulation, by instance in Gysela code \cite{Gysela06}. 
Here we just mention that \eqref{eq:elecneutr} is obtained in a close to
equilibrium setting, with an adiabatic hypothesis on the distribution of the
electrons $n_e =n_0 e{-\frac{e \Phi}{T_e}} \approx n_0 \left( 1 + \frac{e
\Phi}{T_e}\right)$, and an hypothesis of adiabatic response of the ions on the
gyro-circle which gives rise to the $\Phi \ast H_T$ term. As usual in
quasi-neutral equation, we have no good a priori estimates on
the regularity of $E = - \nabla \Phi$.

\medskip
Remark that even if that equation
\eqref{eq:gyroFP2D} is derived from a Vlasov
model (A rigourous derivation of a more general $3D$ model is performed for
fixed field $E$ in section \ref{sect:dif}), it is of ``fluid'' nature. In fact
there is no transport in the remaining of the velocity variable $u$, and the
position of the gyro-center is transported by the eletric drift $(J^0_u
E)^\perp$. So that the equation is similar to the $2D$ Navier-Stokes equation
written in vorticity. More precisely, we have a family of fluid model depending
on a parameter $u$, which are coupled thanks to diffusion in the $u$ variable,
and by the closure used for $E$ described below. 

  Moreover, we will prove in the following that thanks to gyro-average $J^0_u$,
equation has the same regularity than the NS2D equation in vorticity. In fact, the force field $J^0_u \nabla_x \Phi$ belongs naturally to $H^1$ if $f \in L^2$ with some weight.  That is why we obtain the same result that are known about the NS2D equation : global existence and short time uniqueness and stability. However, our model present an additional difficulties which is the lose of regularity for small $u$. In fact, for small value of $u$ the $H^1$ bound (in $x$ only) of $J^0_u \nabla_x \Phi$ explodes.  

\medskip
To state our reuslts properly, we will need the following definitions and notations : 
\begin{itemize}

\item In the sequel, the letter $C$ will design a numerical constant, that may change form line to line. Unless it is mentioned, such constants are independent of everything.

\item $L^2_u(\Omega)=L^2(\Omega, u dx  du)$ is the space of square integrable
functions with respect to the measure $u dxdu$.

\item We shall use various norm on $\T^2$ or on $\Omega$. To avoid confusion, we
will use the following convention. All the norm performed on the whole $\Omega$
will have their weight with respect to $u$ as additional indice. By instance
$\| \cdot\|_{2\pi u}$, $\| \cdot \|_{H^1_{2\pi u(1+u^2)}}$. All the norms
without any indices are norm on $\T^2$ only.

\item For any weight function $k : \R^+ \mapsto \R^+$, the norm $\| \cdot \|_{2,m}$
is defined for any function $f$ on $\Omega$ by
\[
 \| f \|_{2,k} = \left( \int \| F(\cdot,u)\|_2 k(u) \,du \right)^{\frac12}
\]
\item The most usefull weights will be $m(u) = 2 \pi u (1+ u^2)$ and $\tm(u)=1+u^2$. 

\item We change a litlle bit the duality used to define distributions in the
following definition

\begin{definition}
Using distributions with the weight $u$ means that duality is performed as
\[
\langle f, g \rangle_u = \int f g \, dx_g dv_\para u du \, .
\]
\end{definition}
This definition may seem a little artificial because the simple
definition of derivative with respect to $u$,
is not valid. Instead,
\[
\langle \partial_u f, g \rangle_u  = - \langle f, \partial_u g \rangle_u -
\langle \frac{f}{u}, g \rangle_u  \, .
\]
However, this weight respects the underlying physics ($u$ is in fact the 1D norm of a 2D velocity variable) and has many advantages. For
instance the operator $(1/u) \partial_u(u\partial_u )$ is
self-adjoint with this weight. 
\end{itemize}

\medskip
Our precise result are the following. We prove global existence under the hypothesis $\|f_i \|_{2,m}  <+ \infty$.

\begin{theorem}\label{thm:existence}
Let $f_i$ satisfy  $\|f_i \|_{2,m}    <+ \infty$.
Then there exists at least one weak solution $f \in L^{\infty }( \R^+,L^2_u(\Omega ))\cap L^{2}(\R^+, H^1_u(\Omega ))$ to \eqref{eq:gyroFP2D}-\eqref{eq:elecneutr} with initial condition $f_i$, which also satisfies for any $t>0$
\[
\|f(t)\|_{2,u}^2 + \nu \int_0^t \|(\nabla_x,\partial_u) f \|^2_{2,u} \,ds \leq \|f_i \|_{2,u}  \,,
\]
and all the a priori estimates of the previous section (Lemma \ref{lem:u-moment}, \ref{lem:disp}, \ref{lem:grad_xf}) if their initial hypothesis are satisfied. 

\end{theorem}

And we prove short time uniqueness and stability under the additional hypothesis $\|\nabla_x f \|_{2,m} < + \infty$. 

\begin{theorem}\label{thm:uniqueness}
Let $f_i$ satisfy  \[
\| f_i \|_{2,m}  + \| \nabla_x f  \|_{2,m} <+\infty \,.
\]
Then the positive time $\tau^\star$ defined in Lemma \ref{lem:grad_xf} is such that
the weak solution to \eqref{eq:gyroFP2D}-\eqref{eq:indata} , defined in Theorem \ref{thm:existence}, is unique on $[0,\tau^\star]$ . 

Moreover, that solution is stable on that interval of time in the following sense.  Assume that $(f^n)_{n \in \N}$ is a family of solutions given by theorem \ref{thm:existence} with initial conditions $f^n_i$ satisfying
\[
\lim_{n \rightarrow +\infty} \| f^n_i -f_i \|_{2,m} =0 \,, \quad \text{and } \quad \sup_{n \in \N} \|f^n_i\|_{L^2_m(L^4)} < + \infty \,.
\]
Then
\[
\lim_{n \rightarrow +\infty} \sup_{t \in [0,\tau^*]} \| f^n(t) -f(t) \|_{2,m}  = 0 \,.
\]
\end{theorem}

This local result has some more consequence when relating it to the
bound on $\nu \int_0^T \|\nabla_x f\|_{2,m}^2 \,dt \leq \| f_i\|_{2,m}^2
+ C(T) \|f_i\|_2^2$ satisfied by any solution in the sense of
\ref{thm:existence}. The last bound implies that 
$\|\nabla_x f\|_{2,m}$ is almost surely finite. The local result
implies more : that the norm of gradient may blow up only on a closed and
negligeable set, of $4/5$-capicity zero..... 

%

%
%

\medskip
 In the next section the diffusive
operator of \eqref{eq:gyroFP2D} is rigourously derived from  a linear
Vlasov-Fokker-Planck equation in the limit of large magnetic field. In the third
section, some useful lemmas are established, proving regularizing properties of
the gyro-average, global preservation of some weighted norm of $f$, the short
time preservation of the $u(1+u^2)$-moment of $\nabla_x f$ by the system
\eqref{eq:gyroFP2D}-\eqref{eq:elecneutr}), and  controlling the electric
potential by the physical density. This allows to prove the global existence 
(Theorem \ref{thm:existence}) of solutions to the Cauchy problem in the fourth section
and their short time uniqueness and stability
(Theorem \ref{thm:uniqueness}) in the fifth section. Finally some useful properties of the first
Bessel function $J^0$ are proven in the appendix.

%
%
%
%
\sectionnew{Derivation of the gyro-Fokker-Planck operator} \label{sect:dif}

In that section, we rigorously justify the form of the Fokker-Planck appearing in the right-hand side of \eqref{eq:gyroFP2D}. The usual collision operator for plasmas is the nonlinear Landau operator originally introduced by Landau \cite{Landau}. Because of its complexity, simplified collision operators have been introduced. An important physical litterature exists on the subject, also in the gyro-kinetic case (See \cite{Brizard04} and the references therein). In this paper we choose the simplest possible operator possible, namely a linear Fokker-Planck operator. The reasons of this choice are :

 - Its simplicity will allow to focus on the other difficulties of the model,
 
 - The fact that physicists studying gyro-kinetic models for the core of the plasma mainly assume that the dynamics stays close to equilibrium, in which case a linear approximation of the collision operator is relevant.
 
 - The aim of the paper is not a precise description of collisions. In fact, even if they exists in tokamaks, being needed to produce energy, their effect is small compared to the turbulent transport. However, we are interested by their regularizing effect, since the electro-neutrality equation \eqref{eq:elecneutr} do  not provide enough regularity to get a well-posed problem. This is a major difference to the Poisson equation setting.

\medskip
We start from a simple model for a $3D$ plasma, i.e. a linear Vlasov-Fokker-Planck equation with {an external electric field}, an external uniform magnetic field and linear collision and drift terms, and obtain in the limit of  large magnetic field  a 3D (in position) equation analog to \eqref{eq:gyroFP2D}. In particular, we show  that a usual linear Fokker-Planck term on the speed variables turns into an equation with diffusion terms both in space and Larmor radius variables in the limit. 

Precisely, for any small parameter $\ep >0$ we study the distribution
$f_\ep(t,x,v)$ of ions submitted to an exterior electric field $E(t,x)$
(independent of $\ep$) and an uniform magnetic field $B_\ep=(1/\ep,0,0)$. We
also model collisions (with similar particles and the others species) by a
simple linear Fokker-Planck operator. To avoid any problem with possible
boundary collisions, which are really hard to take into account in gyro-kinetic
theory, we assume that $(x,v) \in \T^3 \times \R^3$, where $\T ^3$ is the 3D
torus. When the scale length of all the parameters are well chosen (in
particular the length scale in the direction perpendicular to the magnetic field
should be chosen of order $\ep$ times the length scale in the parallel
direction, we refer to our previous work  \cite{GheHauNou09} for more details on
the scaling), the Vlasov equation $f_\ep$ satisfies is
\begin{equation} \label{eq:vlares}
\frac{\partial f}{\partial t} +   v_{_{\parallel}} \partial_{x_{_{\parallel}}} f + E \cdot \nabla_v f  +\frac{1}{\ep}(v_\perp \cdot \nabla_{x_\perp} f + v^\perp \cdot \nabla_{v_\perp} f )= \diver_v(\beta v f_\ep) + \nu \Delta_v f_\ep \,,
\end{equation} 
where $\beta,\nu$ are two positive parameters, the subscript $\parallel $ (resp. $\perp$) denotes the projection on the direction parallel (resp. on the plane perpendicular) to $B$, and the superscript $\perp$ denotes the projection on the plane perpendicular to $B$ composed with the rotation of angle $\pi/2$. In others words if $v=(v_1,v_2,v_3)$,
\[
v_\perp  = (v_1,v_2,0), \quad v_\para =(0,0,v_3), \quad v^\perp = (-v_2,v_1,0) \, .
\]

\medskip
The next results require the additional notation,
\begin{equation} \label{eq:J0tilde}
\tilde{J}^0_u g(x_g,\rho_{_L},v_{\parallel}) = \frac{1}{2\pi} \int_0^{2\pi} g(x_g + u e^{i\varphi_c}, \rho_{_L} e^{i(\varphi_c-\frac{\pi}{2})} + v_{\parallel} e_{\parallel}) \,d\varphi_c \, ,
\end{equation}
which is a gyro-average performed in phase space, that will be used as an initial layer to adapt the initial condition to the fast Larmor gyration.

\begin{theorem} \label{thm:FPgyro}
Let $E \in L^\infty_t(L^2)$ and $f_\ep$ be a family of solutions to equation \eqref{eq:vlares} with initial condition $f_i \in L^2$ satisfying \mbox{$\sup_t \| f_\ep(t)  \|_2 \leq \|f_i\|_2$}. Then the family $\bar{f}_\ep$ defined by
\[ \label{eq:cov}
\bar{f}_\ep (t,x_g,v) = f(t,x_g + v^\perp,v)
\]
admits a subsequence that converges in the sense of distributions towards a function $\bar{f}$ depending only on \mbox{$(t,x_g,u=|v|,v_\para)$} and solution to
\begin{equation} \label{eq:FPlim} 
\begin{split}
\partial_t \bar{f}   + v_{\parallel} \, \partial_{x_{\parallel}} \bar{f} + J^0_u E_{\parallel} \, \partial_{v_{\parallel}} \bar{f} + & (J^0_u E)^\perp \cdot \nabla_{x_g} \bar{f} =   \\ 
& \beta( v_\para \partial_{v_\para} \bar{f}  + u \partial_u \bar{f} + 3 \bar{f}) + \nu \left( \Delta_{\xg_\perp} \bar{f} + \frac1 u  \partial_u ( u \partial_u \bar{f} ) \right)\, ,
\end{split}
\end{equation} 
in the sense of distributions with the weight $u$, with the initial condition $\tilde{J}^0_u(f^0)$.
\end{theorem}

\begin{remark}
The reason for the change of variables is that the $1/\ep$-term in equation \eqref{eq:vlares} induces a very fast rotation in the perpendicular direction both in the $x$ and $v$ variables,
\[
v(t) = v^0 e^{i t/\ep} \, ,  \qquad x(t) = x^0 + v^{0 \perp} + v^0 e^{i( t/\ep - \pi/2)} \, .
\]
But in the gyro-coordinates this fast rotation is simply described by a rotation in $v$,
\[
v(t) = v^0 e^{i t/\ep} \, ,  \qquad x_g(t) = x_g^0.
\]
\end{remark}
\begin{remark}
The final diffusion appears in all dimensions except the $\xg_\para$ one.  It does not mean that there is no regularization in that direction. Indeed, the models have diffusion in $v_\para$, which after some time regularize in the $\xg_\para$ direction. This mechanism is well known for the Fokker-Planck equation (see for instance \cite{bouchut}).  However, we are not able to prove this phenomena in the non-linear setting because the electric field of the model lacks regularity.  This is the reason why we will only study the $2D$ model. 
\end{remark}

\begin{proof}[Proof of Theorem \ref{thm:FPgyro}]
We proved in a previous work \cite{GheHauNou09} that, provided $f^0 \in L^2$ and
$E \in L^1_t(W^{1,2}_x)$, a subsequence of $f_\ep$  solutions of
\eqref{eq:vlares}  without the collision term converges towards a solution of
\eqref{eq:FPlim} without the collision term.
In order to simplify the presentation, we will neglect the electric field and the parallel translation terms.  To obtain the result in full generality, the only thing to do is to add the argument given in our previous work to the one given below.  For the same reason, we shall also not treat the problem of initial conditions. 

So consider the above Vlasov Fokker-Planck equation without electric force field and parallel translation,
\begin{equation} \label{eq:vlaFP}
\partial_t f  +\frac{1}{\ep}(v_\perp \cdot \nabla_{x_\perp} f + v^\perp \cdot \nabla_{v_\perp} f ) =  \diver_v(\beta v f) + \nu \Delta_v f\, .
\end{equation} 
The first step is to use the change of variables $(x,v) \rightarrow (x_g=x+v^\perp,v)$. Since
\begin{eqnarray*}
\nabla_v f &  = & \nabla_v \bar{f} - \nabla^\perp_\xg \bar{f},  \\
\Delta_v f & = & \Delta_v \bar{f} + \Delta_{\xg_\perp} \bar{f} - 
2 \nabla_v  \cdot \nabla_\xg^\perp \bar{f}, \\
\nabla_v \cdot (v f) & = & v \cdot \nabla_v \bar{f} + 3 \bar{f} - v \cdot \nabla_\xg^\perp \bar{f} \, ,
\end{eqnarray*}
equation \eqref{eq:vlaFP} becomes 
\begin{equation} 
\partial_t \bar{f}_{\ep}  +\frac{1}{\ep}v^\perp \cdot \nabla_v\bar{f}_{\ep} =  -\beta \Big( v\cdot \nabla _v\bar{f}_{\ep} +3\bar{f}_{\ep} -v\cdot \nabla _{x_g}^{\perp }\bar{f}_{\ep}\Big) \\
+ \nu \Big( \Delta_v \bar{f}_{\ep}+\Delta _{x_g}\bar{f}_{\ep}-2\nabla _v\cdot \nabla _{x_g}\bar{f}_{\ep}\Big) \, .
\end{equation} 
By hypothesis $\bar{f}_\ep $ is bounded in $L_t^\infty(L^2_{x;v})$. Therefore, at least a subsequence of  $(\bar{f}_\ep )$ converges weakly to some $\bar{f} \in L_t^\infty(L^2)$. Passing to the limit in \eqref{eq:vlaFP}, it holds that
\[
v^\perp \cdot \nabla _v\bar{f} = 0 \, ,
\]
since all the other terms are bounded. For $v=(u e^{i\varphi},v_\para)$ where $\varphi$ is the gyro-phase, the previous equality means that $\bar{f}$ is independent of the gyro-phase (in the sense of distribution and thus as a $L^2$ function). 

Equation \eqref{eq:vlaFP} tested against a smooth function $g$ independent of the gyro-phase writes
\begin{equation} \label{eq:FP2}
\int \bar{f}_\ep \left( \partial_t g    -  \beta( v \cdot \nabla_v g  - v \cdot \nabla_\xg^\perp g) - \nu (\Delta_v g + \Delta_{\xg_\perp} g - 
2 \nabla_\xg^\perp  \cdot \nabla_v g) \right) \, dx_gdv = 0 \, .
\end{equation} 
We may also pass to the limit when $\ep $ tends to zero in this equation and obtain that the same equality holds for $\bar{f}_\ep$ replaced by $\bar{f}$, considered as a function defined on $\T^3\times \R^3$.

For the change of variable $v=(u e^{i\varphi}, v_\para)$, 
\[ \nabla_v g =(e^{i \varphi}  \partial_u g + i e^{i \varphi}  \partial_\varphi g , \partial_{v_\para}). \] 
Hence, for any function $g$ independent on the gyrophase $\varphi$, it holds that
\begin{eqnarray*}
&&\Delta_\vg g  = \partial^2_{v_\para} g + \frac{1}{u}\partial_u (u \partial_u g) \, , \\
 &&(\nabla_\xg^\perp \cdot \nabla_\vg) g =  \nabla_\xg^\perp \cdot ( e^{i\varphi} \partial_u g ) =  e^{i\varphi} \cdot \nabla_\xg^\perp \partial_u g, \\
 && v \cdot \nabla_v g = v_\para \partial_{v_\para} g + u  \partial_u g \, .
\end{eqnarray*}
The other terms appearing in \eqref{eq:FP2} remain unchanged. Then, 
\begin{equation} \label{eq:FPuphi}
\begin{split}
\int \bar{f} \Big( \partial_t g    -  \beta( v_\para \partial_{v_\para} g + u  \partial_u g &  - u e^{i \varphi} \cdot \nabla_\xg^\perp g)- \nu (\partial^2_{v_\para} g + \frac{1}{u}\partial_u (u \partial_u g) \\ 
& + \Delta_{\xg_\perp} g - 
 2 e^{i\varphi} \cdot \nabla_\xg^\perp \partial_u g) \Big) \, dx_g dv_\para 2\pi u du d \varphi = 0 \, .
\end{split}
\end{equation} 
Since $\bar{f}$ is independent of $\varphi$, performing the integration in $\varphi$ first makes the term containing $\varphi$ vanish. So the function $\bar{f}$ of the five variables $(x_g,u,v_\para)$ satisfies
\begin{equation} \label{eq:FPuphi2}
\int \bar{f}  \Big( \partial_t g    -  \beta(  v_\para \partial_{v_\para} g +   u  \partial_u g  ) 
  -\nu (\partial^2_{v_\para} g +  \frac{1}{u}\partial_u (u \partial_u g) + \Delta_{\xg_\perp} g
) \Big) \, dx_g dv_\para u du = 0 \, .
\end{equation} 
It exactly means that $\bar{f}$ satisfies the equation 
\begin{equation}
\partial_t \bar{f}   =   \beta( v_\para \partial_{v_\para} \bar{f}  + u \partial_u \bar{f} + 3 \bar{f}) + \nu \left( \partial^2_{v_\para} \bar{f} + \Delta_{\xg_\perp} \bar{f} + \frac1 u  \partial_u ( u \partial_u \bar{f} ) \right)\, ,
\end{equation} 
in the sense of distributions with weight $u$. It is the equation \eqref{eq:FPlim} without parallel transport nor electric field.
\end{proof}


If we look at solutions of this equation invariant by translation in the direction of $B$, we exactly get the 2D-model announced in the introduction. In fact, if $\bar{f}$ is a solution of \eqref{eq:FPlim}, then
\[
f(t,x,u)= \int \bar{f}(t,x,u,v_\para ) \,dv_\para \, 
\]
is a solution of  \eqref{eq:gyroFP2D}. Such an assumption on $f$ is reinforced by experiments and numerical simulations, where it is observed that the distribution of ions is quite homogeneous in $x_{_\parallel}$.

\sectionnew{Some useful lemmas}
\label{sect:apriori}

We prove here some a priori estimates useful for the proof of our theorem.
In order to simplify the proof of some of the following Lemmas, we sometimes uses the following  formulation of \eqref{eq:gyroFP2D} with the genuine two-dimensional velocity variable. 
Denote by $\tf (t,x,\au) = f(t,x,|\au|)$, $\au \in \R^2 $. It is solution (in the sense of
distribution with usual duality) of the following equation with $4D$ in space
and velocity variables
\begin{equation} \label{eq:VFP4D}
\partial_t \tf -\nabla _x^\perp (J^0_{|\au|}\Phi )\cdot \nabla _x \tf = \nu ( \Delta _x \tf+ \Delta_\au \tf ) + \beta ( 2 \tf  +  \au \cdot \nabla_\au \tf) .
\end{equation}
Heuristically, radial in $\au$ solution of equation \eqref{eq:VFP4D} is a solution of \eqref{eq:gyroFP2D}.  We can state for instance a precise Lemma in the case where $\phi$ is fixed and smooth.  

\begin{lemma} \label{lem:3Deq4D}
For a fixed smooth potential $\Phi$, $f$ is the unique solution of
\eqref{eq:gyroFP2D} with initial condition $f_i$ if and only if $\tf$ is
the unique solution of \eqref{eq:VFP4D} with intial condition $\tf_i$.
\end{lemma}

{ \bf Proof of the Lemma \ref{lem:3Deq4D} }:
The proof relies on the uniqueness of the solution to \eqref{eq:VFP4D} (See
\cite{Ladyzen} and the conservation of the radial symmetry of the solution.
\cqfd


%

%
%
\subsection{Regularizing properties of the gyro-average.}

  In this section, some regularizing property of the gyro-average operato are
proven. They are based on the fact that $\hat{J^0} \sim k^{-\frac12}$ for large
$k$ (the precise bound are proved in \ref{App:A}), which implies that $J^0  $
maps $H^s$ onto $H^{s+\frac{1}{2}}$. It is important since the formula \eqref{eq:elecneutr} giving the gyro-averaged
potential in terms of the distribution $f$ involves two gyro-averages, and thus
a gain of one derivative for the gyro-averaged potential w.r.t. $f$. However,
the regularizing properties of $J^0_u$ are bad for small $u$, which raises
difficulties. 

The first lemma of this section gives the regularity of the
gyro-averaged potential in term of the potential $\Phi$. The second one gives
the regularity of the density $\rho$ in terms of the distribution $f$. We will
need the two following definitions before stating it.

\medskip
\begin{definition} Let $f$ be a measurable function defined on $\Omega$. Denote
by
	\[
	\|f\|_{L^2_m(H^s)} = \left( \int  \| f(\cdot,u)\|^2_{H^s} m(u) 
\,dw\right)^{\frac12}
	\]
	the norm with the weight $m(u) = 2\pi u(1+u^2)$. \\

For any $U>0$, let $F$ be a measurable function defined on $\Omega_U = \T^2 \times [0,U]$. Denote by
	\[
	\| F \|_{H^1_U} = \left(   \int_{\T^2} \! \int_0^U \Big( |f|^2 + |\nabla_x f |^2 + |\partial_u f|^2 \Big)  2\pi u \,du\right)^{\frac12}.
	\]
\end{definition}
The lemmas stating the regularity of $\Phi$ and $\rho$ are the following.
%
\begin{lemma} \label{lem:phireg}
For any $s \in \R$, $u >0$ and $\Phi$ with $0$-mean, it holds that
\begin{eqnarray*}
& i) & \| J^0_u\Phi \|_{H^s}   \leq   \| \Phi \|_{H^s}, \\
&ii) & \| J^0_u\Phi \|_{H^{s+\frac12}}  \leq  \frac{2^{\frac14}}{\sqrt u} \|\Phi \|_{H^s} \, , \\
&iii)& \| \partial_u J^0_u\Phi  \|_{H^s}  \leq  \frac 1 {\sqrt u} \|\Phi \|_{H^{s + \frac12}}.
\end{eqnarray*}
As a consequence, for any  $U>0$,
\begin{eqnarray*}
&iv)& \| J^0_u \Phi \|_{H^1_U} \leq 4 \sqrt U \|\Phi\|_{H^{\frac12}}.
\end{eqnarray*}
\end{lemma}
%
\begin{lemma} \label{lem:rhoreg} For any $s >0$, if  $\int f\,  2\pi u \,dxdu =1$ and $\rho$ is defined by \eqref{eq:defrho}, then 
\begin{equation}
\|\rho -1 \|_{H^{s+\frac12}}  \leq  2^{\frac14}  \pi \|f \|_{L^2_m(H^s)}.
\end{equation}
\end{lemma}

\underline{Proof of Lemma \ref{lem:phireg} and \ref{lem:rhoreg}.}\\
Denote by $ \hat{\Phi}(k)$ the $k$ Fourier coefficient of $\Phi$. Then
\[
\| J^0_u \Phi \|_{H^s}^2 = \sum_{k=1}^\infty |J^0_u(k)|^2|\Phi(k)|^2 \leq  \sum_{k=1}^\infty |\Phi(k)|^2 = 
\| \Phi \|_{H^s}^2 \,,
\]
using the bound $\|\hat{J^0}\|_\infty \leq 1$ proved in Lemma \ref{lem:boundJ}.
It is the inequality $i)$. For the second inequality, remark that 
\begin{equation} \label{eq:calc}
\quad \frac{1+|k|^2}{1+w^2|k|^2} = \frac1 {w^2} \frac{1 + \frac1{|k|^2}}{1 + \frac1{w^2|k|^2}} \leq \frac2{w^2} \;, \quad \, k\in \Z ^*,
\end{equation}
and use it together with $ii)$ of  Lemma \ref{lem:boundJ} in
\begin{eqnarray*}
\|J^0_u\Phi  \|_{H^{s+\frac12}}^2 & = & \sum_{k \neq 0} |\hat{J^0_u\Phi }(k,u)|^2 (1+|k|^2)^{s+\frac12} \\
& = & \sum_{k \neq 0} |\hat{\Phi}(k)|^2 |\hat{J}^0(|k|u)|^2 (1+|k|^2)^{s+\frac12} \\
& \leq &  \sum_{k \neq 0} |\hat{\Phi}(k)|^2 (1+|k|^2)^s \sqrt{\frac{1+|k|^2}{1+|k|^2 u^2}} \\
& \leq & \frac{\sqrt{2}}{u} \|\Phi\|^2_{H^s}\,.
\end{eqnarray*}
For the third estimate of Lemma \ref{lem:phireg}, remark that
\[
\left( \partial_u \hat{J^0_u\Phi }\right) (k)  =  \partial_u \left( \hat{J^0}(|k|u) \hat{\Phi}(k) \right) = |k| \hat{\Phi}(k) \hat{J^0}'(|k|u) \, 
\]
and use the bound $iii)$ of Lemma  \ref{lem:boundJ}  to get
\begin{eqnarray*}
 |( \partial_u \hat{J^0_u\Phi })(k)| & \leq  &   \sqrt{\frac{|k|}u}  | \hat{\Phi}(k)|\, .
\end{eqnarray*}
From this, we obtain
\begin{eqnarray*}
 \| \partial_u (J^0_u\Phi) \|_{H^s} & \leq & \frac1{\sqrt u} \|\Phi \|_{H^{s + \frac12}}.
\end{eqnarray*}

The point $iv)$ uses the previous inequalities. First remark that the norm $\| \hspace*{0.02in}\|_{H^1_U}
$ is also equal to
\[
\| F \|_{H^1_U}=	\left( \int_0^U \Big(  \| \partial_u F(\cdot,u)\|_{L^2}^2 +  \| F(\cdot,u)\|_{H^1}^2\Big) 2\pi u \,du \right)^{\frac12} \,.
\]
Using this formulation and  $ii)$- $iii)$ leads to
\begin{eqnarray*}
\| J^0_u \Phi \|_{H^1_U}^2 & = &  \int_0^U \Big(  \| \partial_u J^0 \Phi \|_{L^2}^2 +  \| J^0 \Phi\|_{H^1}^2\Big) 2\pi u \,du \\\
& \leq & 2 \pi \| \Phi \|_{H^{\frac12}}^2\int_0^U \frac{1 + \sqrt 2} u   u \,du \leq 16 U \|  \Phi \|_{H^{\frac12}}^2,
\end{eqnarray*}
which gives the desired result and ends the proof of Lemma \ref{lem:phireg}.\\

\underline{Proof of Lemma \ref{lem:rhoreg}.}\\
Denote by $\Frho(k)$ the $k$-th Fourier term of $\rho $ with respect to the space variable, i.e.
\begin{eqnarray*}
\Frho(k) & = & 2\pi \int J^0(|k|w) \hat{f}(k,w) w \,dw.
\end{eqnarray*}
By \ref{lem:boundJ},
\begin{eqnarray*}
|\Frho(k)| & \leq & 2 \pi \int \frac{|\hat{f}|(k,w) w}{(1 + w^2 |k|^2)^{1/4}} \,dw.
\end{eqnarray*}
It follows from \eqref{eq:calc} that for $k \neq 0$, 
\begin{eqnarray*}
(1+|k|^2)^{\frac {2s +1}4} |\Frho(k)| &\leq &2^{\frac{5}{4}} \pi \int_0^\infty |\hat{f}|(k,w) (1+|k|^2)^{\frac s 2}\sqrt w \,dw   \\
&\leq &2^{\frac54} \pi \left(\int_0^\infty |\hat{f}|^2 (k,w) (1+|k|^2)^s w(1+w^2) \,dw  \right)^{1/2} \left( \int_0^\infty \frac{dw}{(1+ w^2)} \right)^{1/2}  \\
&= &2^{\frac14} \pi \left(\int_0^\infty |\hat{f}|^2 (k,w) (1+|k|^2)^s 2 \pi w(1+w^2) \,dw  \right)^{1/2} .
\end{eqnarray*}
Hence, since $\Frho(0) = \int_{\T^2} \rho(x)\,dx =1$ by mass conservation, 
\begin{eqnarray*} 
\|\rho -1\|_{H^{s +\frac12}} & \leq &  2^{\frac14}  \pi \sqrt{\sum_{k \neq 0} \left( \int_0^\infty |\hat{f}(k,w)|^2 (1+|k|^2)^{\frac s 2}  2\pi w(1+w^2) \,dw \right)}  \\
 & \leq &  2^{\frac14}  \pi \left( \int_0^\infty \|f(w)\|_{H^s}^2 2\pi w(1+w^2) \,dw \right)^{1/2} \,,
\end{eqnarray*}
and Lemma \ref{lem:rhoreg} is proved.

\cqfd

\subsection{Control of the potential by the density.}

Denote by $L_T$ the operator that maps any function $\Phi$ on $\T^2$ with
zero mean to $\frac 1T (\Phi - \Phi \ast _x H_T)$ and by $H^s_0(\T^d)$ 
the space of $H^s$ functions with zero mean. This section is devoted to a proof
of the boundedness of $L^{-1}$ from $H^s_0(\T^d)$ onto $H^s_0(\T^d)$. Recall
that in a Fourier setting (See the Appendix of \cite{GheHauNou09} for more
details), the operator $H_T = I -T L_T $ is the multiplication by
\[
\hat{H_T} (k) = \frac{2}{T} \int_0^{+\infty} J^0(ku)^2 e^{- u^2/ T} u\,du.
\]

\begin{lemma} \label{lem:boundL}
The Fourier multipliers $\hat H_T(k)$  satisfy,
\begin{equation} \nonumber
 |1 - \hat{H} (k)| \geq   \frac{|k|^2 T}4 \left( 1 - e^{-\frac 1 {|k|^2 T}}\right)\,, \quad \forall\, k\in \Z ^2 \setminus 
 \{ (0,0)\} .
\end{equation}
As a consequence, the operator $L_T^{-1}$ maps any $H^s_0$, $s\in \R$, into
itself with norm
\begin{equation}
\| L^{-1}\|_{H^s_0} \leq  c_T := \frac 4 {1 - e^{-\frac 1 T }}.
\end{equation}
\end{lemma}
\begin{remark}
Lemma \ref{lem:boundL} shows  that $\| L^{-1}\| $ is bounded for small $T$, and of order $T$ for large $T$, the physical case of interest.\\
The boundedness of the spatial domain is essential.  When defined on 
the whole space $\R^2$ rather than on the torus, the operator $L^{-1}$ is not 
bounded. Its norm explodes in the low frequency range.
\end{remark}
\hspace*{0.01in}\\
\underline{Proof of the Lemma \ref{lem:boundL}}
Two bounds on $J^0(l)$ are used, namely one of the bounds of 
Lemma \ref{lem:boundJ} for $l\geq 1$ and the following 
bound given by the Taylor expansion of $J^0$ near $0$ for $l\leq 1$,
\[
 0 \leq (J^0(l))^2 \leq  1 - \frac{ l^2}{4} \,, \quad \text{if } 0\leq 
l \leq 1 \, .
\]
Consequently,
\begin{eqnarray*}
|\hat{H_T}(k) | & \leq & \frac{2}{T}  \int_0^{\frac 1 {|k|}}  \left( 1- 
\frac{(|k|u)^2}{4}\right) e^{- u^2/ T} u\,du +
\frac{\sqrt 2}{|k|T} \int_{\frac 1 {|k|}}^\infty e^{- \frac{u^2} T} \,du \\
& \leq & 2  \int_0^{w}  \left( 1- 
\frac{x^2}{4 w^2}\right) e^{- x^2} x\,dx 
+ \sqrt 2 w \int_{w}^\infty e^{-x^2} \,dx  \\
& \leq & 1 - \frac{3}{4}e^{-w^2}  - \frac{1}{4w^2} ( 1 - e^{-w^2} )  + 
\sqrt 2 w  \int_w^\infty e^{-x^2}\,dx,
\end{eqnarray*}
where $w=(|k|\sqrt{T})^{-1}$. Now, using the bounds $2^{-\frac12} < \frac34$ and
\[
w \int_w^\infty e^{-x^2}\,dx  \leq \int_w^\infty x  e^{-x^2}\,dx  = 
\frac{e^{-w^2}}{2} \,,
\]
it holds that
\[
1 - |\hat{H_T}(k) |  \geq   \frac{1}{4w^2} ( 1 - e^{-w^2} ).
\]
This is the first claim of lemma \ref{lem:boundL}
The function of $w$ in the right-hand side of the previous inequality on $|\hat{H}(k) |$ is decreasing and goes from $\frac14$ at $0$ to $0$ at $+\infty$. Consequently its minimal value are obtained for large $w$ i.e. for small $|k|$, namely $|k|=1$. Precisely ,
\[
1 - \sup_{k \neq 0} |\hat{H_T}(k) |  \geq   \frac{T}{4} \left( 1 - e^{-\frac 1 T
} \right).
\]
Since the Fourier representation of $L_T^{-1}$ is the multiplication by $ T(1 -\hat{H}(k))^{-1}$ we obtain that in any $H^s_0$, $s\in \R $,
\[
\| L_T^{-1} \|_{H^s_0}  = \sup_{ k \neq 0} \frac T {|1 -\hat{H}(k)|} \leq \frac 4 {1 -
e^{-\frac 1 T }},
\]
which is the desired result. \cqfd

%
%
\subsection{Propagation of $L^2_m$ and $L^2_m(L^4)$ norms of $f$.} \label{sec:moment}

The two following lemmas will be useful in the sequel.  

\begin{lemma}\label{lem:u-moment}
Assume that $\| f_i\|_{2,u}^2 < +\infty $. Then, any solution of
\eqref{eq:gyroFP2D} and \eqref{eq:indata}, for regular potential $\phi$, satisfies
\[
 \forall\, t>0, \quad \| f(t) \|_{2,u} \leq e^{\beta t} \| f_i \|_{2,u} \,.
\]
Assume moreover that $\| f_i\|_{2,m} < +\infty $.
Then any solution $f$ satisfies 
\begin{equation}\label{eq:u-moment}
\| f(t) \|^2_{2,m} \leq \| f_i \|^2_{2,m} + (2\nu +\beta)  \frac{e^{2 \beta t} -1 } \beta \| f_i \|^2_{2,2\pi u}  \,.
\end{equation}
with the convention that $\frac{e^{2 \beta t} -1 } \beta = 2t$ if $\beta =0$. 
\end{lemma}

\begin{lemma} \label{lem:disp}
Assume $\| f_i \|_{L^2_m(L^4)}
 < +\infty$ and $f$ is a solution of \eqref{eq:gyroFP2D} with initial
condition $f_i$ with a regular potential $\phi$. Then $f$ satisfies
\begin{equation} \label{eq:disp}
\| f(t) \|_{L^2_m(L^4)} \leq e^{(\beta +  2 \nu) t } \| f_i \|_{L^2_m(L^4)} 
\end{equation}
\end{lemma}

\begin{remark}
A more careful analysis will show that
\[
\| f(t) \|^2_{L^2_m(L^4)} \leq \| f_i \|^2_{L^2_m(L^4)} + (2\nu +\beta)  \frac{e^{2 \beta t} -1 } \beta \| f_i \|^2_{L^2_{2\pi u}(L^4)} \,,
\]
but the simple estimate of Lemma \ref{lem:disp} will be sufficient. 
\end{remark}

\underline{Proof of Lemma \ref{lem:u-moment}}]
 Multiply equation \eqref{eq:VFP4D} by $\tf$.  Using the notations
\begin{equation}
u=|\au|,\quad g(t,u) = \frac12 \|\tf(t,\cdot,\au)\|^2_2,
\end{equation}
and integrating in the $x$ variable leads to
\begin{equation}
\partial_t g  -  \nu  \Delta_\au g  = -  \|(\nabla_x,\nabla_u)  \tf(t,\cdot,u)\|^2_2 +  \beta (  4g  + \au \cdot \nabla_\au g ) .
\end{equation}

Multiply the previous equation by $k(\au)$,  where $k$ is a smooth function on
$\R^2$  with compact support and integrate in the velocity variable $\au$ leads
to
\begin{equation*} \begin{split}
\partial_t \left( \int g(t,u) k(\au) \,d\au \right) +  \int
\|(\nabla_x,\nabla_\au) &  \tf(t,\cdot,u)\|^2_2 k(\au) \,d\au  
  = \\
   &    \int (\nu \Delta_\au k(\au) + 4 \beta k(u) - \beta \diver (k(\au)\au) 
)g(t,u) \,d\au .
\end{split} \end{equation*}
By approximation, this is still true for functions $k$ with unbounded supports.
For $k(\au)=1$, 
\[
\partial_t \left( e^{- 2 \beta t} \int g(t,u)  \,d\au \right) = - e^{- 2 \beta t} \int \|(\nabla_x,\nabla_u) \tf(t,\cdot,u)\|^2_2  \,d\au  \leq 0.
\]
Coming back to the 1D  original quantities, it means that
\begin{equation} \label{eq:calc2}
\| f(t) \|_{2,u} \leq e^{\beta t} \|f_i\|_{2,u}.
\end{equation}

For $k(\au) = \tm(u)$,  then $\Delta k = 4$
and 
\[
4 \tm(u) -  \diver (\tm(u) \au) = 2 \tm(u) - \tm'(u) u = 2 \,.
\]
Therefore,
\[
\int g(t,u) \tm(u) \,d\au  \leq \int g(0,u) \tm(u) \,d\au + 2(2\nu +\beta)  \int_0^t \int   g(s,u) \,d\au \,ds
\]
Or in other words
\[
\| f(t) \|^2_{2,m} \leq \| f_i \|^2_{2,m} + 2(2\nu +\beta)  \int_0^t \| f(s) \|^2_{2,u} \,ds \,.
\]
Using the bound of equation \eqref{eq:calc2}, we get
\[
\| f(t) \|^2_{2,m} \leq \| f_i \|^2_{2,m} + (2\nu +\beta)  \frac{e^{2 \beta t} -1 } \beta \| f_i \|^2_{2,u}  \,.
\]
with the convention that $\frac{e^{2 \beta t} -1 } \beta = 2t$ if $\beta =0$. \cqfd

\underline{Proof of Lemma \ref{lem:disp} }
  In order to simplify the presentation, we will first performed calculation without justifying every integration by parts and division. But once we obtain an a-priori result, we will explain the small adaptation needed to make it rigorous. 
 First, we denote 
\begin{eqnarray*}
\alpha (t,\au)= \int | \tf(t,x,\au)| ^4 dx = \| f \|_4^4 \,, \qquad \gamma(t,\au) = \int |\tf|^2|\nabla_\au \tf|^2\,dx 
\end{eqnarray*}
Multiply equation \eqref{eq:gyroFP2D} by $3 \signe(f) |f|^3$ and
integrating with respect to $x$ leads to
\begin{equation} \label{eq:alpha}
\partial_t \alpha  =  - 12 \nu  \int f^2 |(\nabla_x,\nabla_\au) f |^2 \,dx +
\nu \Delta_\au \alpha  + 8 \beta \alpha +  \beta \au \cdot \nabla_\au \alpha \,.
\end{equation}
Hence, dividing by $\sqrt \alpha$
\[
\partial_t \sqrt \alpha  = \frac { \partial_t \alpha }{2 \sqrt \alpha} \leq  - 6 \nu \frac \gamma {\sqrt \alpha}  + \frac{\nu\Delta_\au \alpha}{2 \sqrt \alpha}
+ 4 \beta \sqrt \alpha + \beta \frac{\au \cdot \nabla_\au \alpha}{2 \sqrt \alpha}
\]
Now, we multiply by $\tm(u)$, integrate with respect to $\au$ 
\[
 \partial_t \left( \int \sqrt \alpha \, \tm(u) d\au \right)  \leq   - 6 \nu \int \frac \gamma{\sqrt \alpha}\, \tm(u)\,d\au + \frac \nu 2 \int \frac{\Delta_\au \alpha }{\sqrt \alpha} \tm(u) d\au 
 + 4 \beta \int  \sqrt \alpha  \,\tm(u)\,d\au + \beta  \int \frac{\au \cdot \nabla_\au \alpha}{2 \sqrt \alpha} \tm(u)\,d\au
\]
With the help of some integration by parts, we get that
\begin{eqnarray*}
\int \frac{\au \cdot \nabla_\au \alpha}{2 \sqrt \alpha} \tm(u)\,d\au  & = & - 2 \int \sqrt \alpha (\tm(u) + u^2) \,d\au \\
\int \frac{\Delta_\au \alpha }{\sqrt \alpha} \tm(u) d\au  & = & - 2 \int \frac{\au \cdot \nabla_\au \alpha}{\sqrt \alpha} \,d\au +   \int \frac{|\nabla_\au \alpha|^2}{2 \alpha^{\frac32}} \tm(u) \,d\au \\
& = &  8 \int \sqrt \alpha \,d\au + \int \frac{|\nabla_\au \alpha|^2}{2 \alpha^{\frac32}} \tm(u) \,d\au  \,.
\end{eqnarray*}  
Thanks to that, the previous  inequality simplify in
\[
 \partial_t \left( \int \sqrt \alpha \, \tm(u) d\au \right)  \leq   - 6 \nu \int \frac \gamma{\sqrt \alpha}\, \tm(u)\,d\au  + \frac  \nu 4 \int \frac{|\nabla_\au \alpha|^2}{\alpha^{\frac32}} \tm(u) \,d\au + 2 (\beta + 2 \nu)  \int  \sqrt \alpha \,d\au
 \]
Next we can estimate $|\nabla_\au \alpha|$ in terms of $\gamma$. In fact by H\"older inequality
\begin{eqnarray*}
\nabla_\au \alpha &  = & \nabla_\au \left( \int f^4 \,dx \right) = 4 \int f^3 \nabla_\au f  \,dx \\
|\nabla_\au \alpha|^2 & \leq & 16 \left( \int f^4 \,dx \right)  \left( \int f^2 | \nabla_\au f|^2  \,dx\right)   = 16 \alpha \gamma
\end{eqnarray*} 
So that the second term in the right hand side of the previous inequality is controlled up to a constant to the first one. We precisely get
\begin{equation} \label{eq:ineqdiff}
 \partial_t \left( \int \sqrt \alpha \, \tm(u) d\au \right)  \leq   - 2 \nu \int \frac \gamma{\sqrt \alpha}\, \tm(u)\,d\au   + 2 (\beta + 2 \nu)  \int  \sqrt \alpha \,d\au
\end{equation}
From which we conclude easily. 

In the previous calculation, we have not justified all the integrations by part. To make the argument rigorous, a possibility is to choose a smooth function $\xi_1$ from $\R^+$ into $[0,]$ such that 
$\xi(u) =1$ if $u \in [0,1]$ and $\xi(u) = 0$ if $u \in [1,+\infty)$, and define for all $U >0$ $\xi_U(u) = \xi \left( \frac u U\right)$.  Remark that $ |U \xi_U' |_\infty \leq |\xi'|_\infty$ and  $|U^2 \xi_U'' |_\infty \leq |\xi''|_\infty$. Then, we performed the previous calculation with the weight $\tm_U = \tm \xi_U$, 
we obtain an inequality very similar to \eqref{eq:ineqdiff} 
\begin{equation} \label{eq:diffapprox}
 \partial_t \left( \int \sqrt \alpha \, \tm_U(u) d\au \right)  \leq   - 2 \nu \int \frac \gamma{\sqrt \alpha}\, \tm_U(u)\,d\au   + \left[ 2 (\beta + 2 \nu) + \frac C {U^2} \right]  \int  \sqrt \alpha \,d\au
\end{equation}
%
%

from which we get $\| f(t) \|_{L^2_{\tm_U}(L^4)} \leq e^{(\beta +  2 \nu+ \frac C {U^2}) t } \| f_i \|_{L^2_{\tm_U}(L^4)} $, which give the desired result letting $U$ going to infinity. 

The other point not rigorously justified is the division by $\sqrt \alpha$ that may be zero. However, since we have a diffusion equation, it may be proved that for $t > 0$, $\alpha>0$ everywhere.  Or we can use a family of smooth approximation of $\sqrt \cdot $. Or we can say that $\alpha + \ep$ satisfy  \eqref{eq:alpha} with a additional term that has the good sign, so that it will satisfy \eqref{eq:diffapprox}, and we will obtain the desired inequality letting $\ep$ going to zero and then $U$ going to infinity.  It is well justified since the maximum principle applies there so that any solution with non-negative initial condition remains non-negative.

%
%
\subsection{Short time estimate of the $m$-moment of $\nabla _x f$.}
\label{sec:moment-grad_xf}

The following lemma provide is central in the proof of the stability and uniqueness of the solution for short time. 

\begin{lemma} \label{lem:grad_xf}
Assume that $f$ is a solution of the system \eqref{eq:gyroFP2D}-\eqref{eq:indata} satisfying initially $\| \nabla_x f_i\|_{2,m} < +\infty$. Then there exists a constant $C^*$  and a time $\tau^*$ depending on $(T,\nu,\|\nabla_x
f_i\|_{2,m})$, such that
\[
 \|\nabla_x f_i\|^2_{2,m} + \frac \nu 2  \int_0^{\tau^*} \|
(\nabla_x,\partial_u) \nabla_x f \|^2_{2,m} \,dt \leq C^* 
\]
\end{lemma}

We also mention that the result is true if the definition of $\Phi$ in \eqref{eq:elecneutr} is replaced by another definition which still satisfies the bound given in Lemma \ref{lem:rhoreg} and \ref{lem:phireg}.  Precise bounds by below for $\tau^*$ are given at the end of the proof (only in the case $\beta=0$). 
\medskip

\underline{Proof of Lemma  \ref{lem:grad_xf}  :} 
We take the $x$-gradient of equation \eqref{eq:VFP4D}, written in $2D$ in $\au$ (with $u =|\au|$), and obtain
\begin{equation*}
\partial_t \nabla_x \tf - \nabla_x^\perp (J_{u }^0\Phi )
\nabla_{x,x}^2 \tf=  \beta(2\nabla_x \tf +  \au \cdot \nabla_\au (\tf)) + \nu  \Delta_{x, \au} (\nabla_x \tf)  - \nabla
_x(\nabla_x^\perp (J_{ u}^0\Phi ) \nabla_x \tf.
\end{equation*}
If we now multiply by $\,^t \nabla_x \tf$ on the left and
integrate in $x$, the function $h$ defined by $h(t,u) = \frac12 \int |\nabla_x
\tf|^2 \,dx$ satisfies,
\begin{equation}  \label{eq:gevol}
\partial_t g(u)  = \beta (4 g(u) + \au \cdot \nabla_\au g(u)) \nu  \Delta_{\au} g(u) - \nu \| \nabla_{x,\au} \nabla_x
\tf    \|_2^2 - \int  ~^t\nabla_x \tf \, \nabla_x (\nabla_x^\perp
J_u^0\Phi) \, \nabla_x \tf  \,dx.
\end{equation} 
We may also multiply this equation by $\tm(u )= (1+u^2)$ and integrate it in
$\au$. We obtain after that
\begin{equation} \label{eq:gevol2}
\frac12 \partial_t \|\nabla_x \tf \|_{2,\tm}^2   + \nu \| \nabla_{x,\au}
\nabla_x \tf    \|_{2,\tm} ^2  = (4 \nu + 2\beta)  \|\nabla_x \tf \|_{2,u^0}^2  -
\int \int  ~^t\nabla_x \tf \Big( \nabla_x (\nabla_x^\perp J_u^0\Phi) \Big)
\nabla_x \tf  \, \tm(u) \,dx\,d\au.
\end{equation}

To go on, we need to understand a little better the matrix $M(t,x,u) = \nabla_x
(\nabla_x^\perp J_u^0\Phi)$. First remember that $\Phi = L_T^{-1}(\rho-1)$, and
then remark that from there definition, $J^0$ and $L^{-1}$ commute with
derivation in $x$. So that our term may be rewritten $ M = J_u^0 L^{-1}
(\nabla_x (\nabla_x^\perp \rho))$. Using the bound of the Lemma \ref{lem:rhoreg}
and \ref{lem:boundL} we obtain that 
\[
\forall u>0\,, \quad 
\| M(t,u) \|_{H^1} = \| J_u^0 L^{-1} (\nabla_x (\nabla_x^\perp \rho)) \|_{H^1}
\leq \frac{2^{\frac14} c_T} {\sqrt u} \| \rho -1 \|_{H^{\frac 52}} \leq \frac{C
c_T} {\sqrt u} \| f \|_{L^2(H^2_m)} \,.
\]
Moreover, the $H^2_m$-norm of $f$ appears in the right-hand side of
\eqref{eq:gevol}.  So that we may use it to control  $M$. 

With a control on the $H^1$ norm of $M$, we do not get an infinite bound on $M$,
like $\|M\|_\infty \leq C \|f\|_{H^2}$. In that case, we will be able to use a
classical tool to conclude. But this is almost true, we are in a critical case
$(d=2$ and $p=2$) for the Sobolev imbeddings, but we still know that the square
of $M$ is exponentially integrable. Precisely, since $M$ is of
average $0$, we have for all $u >0$ the following Tr\"udinger inequality 
\[
\int_x   e^{\frac{M^2}{6 \|M\|^2_{H^1}}}\,dx \leq 2
\]
We refer to \cite{Moser} for a proof of that result.
To estimate $\int \! \int \,^t \nabla \tf M \nabla \tf (1+u^2) \,dx
du $, we first perform the integral in $x$. For this, we apply the inequality
\[
 ab \leq e^a -b + b \ln b \,, \qquad \forall\, a,b >0,
\]
to $(a,b)= \left( \left(\frac{|M|}{ 6 \|M\|^2_{H^1}}\right)^2 ,\frac{|\nabla
\tf|^2}{\|\nabla \tf\|^2_2} \right) $. The previous inequality comes
from the Legendre transform of $e^a$, as the Young inequality, so that our
application will be a log-exp analog of the H\"older inequalities. We obtain 
\begin{eqnarray*}
 \| M \nabla \tf \|_2 & = & 6 \|M \|_{H^1}
\|\nabla \tf\|_2  \left( \int  \left( \frac{|M|}{ 6
\|M\|^2_{H^1}}\right)^2
\left(\frac{|\nabla \tf|^2}{\|\nabla \tf\|^2_2}\right)^2  \,dx
\right)^{\frac12} \\
& \leq &   \frac{C c_T} {\sqrt u} \| \tf \|_{L^2_\tm(H^2)} \|\nabla \tf\|_{2}  
\left( 2 - 1 + \int 
\frac{|\nabla \tf|^2}{\|\nabla \tf\|^2_{2}} \ln \left( \frac{|\nabla \tf
|^2}{\|\nabla \tf \|^2_2} \right)  \,dx \right)^{\frac12} \,, \\
& \leq & \frac{C c_T} {\sqrt u} \| \tf \|_{L^2_\tm(H^2)} \|\nabla \tf\|_{2}   \left(
1 +  \ln \left( \frac{\|\nabla_x \tf \|^4_4}{\|\nabla_x \tf \|^4_2} \right)  
\right)^{1/2} \,, \\
& \leq & \frac{C c_T} {\sqrt u} \| \tf \|_{L^2_\tm(H^2)} \|\nabla \tf\|_{2}   \left(
1 +  \ln \left( \frac{c_s \|\nabla^2_{x,x} \tf \|^2_2}{\|\nabla_x \tf
\|^2_2} \right)   \right)^{\frac12} \,,
\end{eqnarray*}
where  we have used the Jensen inequality - precisely, for a function $g$ of
integral $1$, $\int g \ln g \leq \ln ( \int g^2)$ - in the last but one line, and
the Sobolev imbedding from $H^1$ into $L^4$ with constant $c_s$ in the
last one. The constant $C$ may change from line to line. Using this
 and Jensen inequality in the previous equation, we get
\begin{equation*}
\begin{split}
\int \! \int | \,^t \nabla & \tf M \nabla \tf|  \, \tm(u) \,dx du  
\leq  \int   \| M \nabla \tf \|_2 \|\nabla \tf \|_2 \tm(u) \,du \,, \\
 & \leq  C c_T | \tf \|_{L^2_\tm(H^2)} \int    \|\nabla \tf\|_2^2   \left( 1
+  \ln \left( {\textstyle \frac{c_s \|\nabla^2_{x,x} \tf
\|^2_2}{\|\nabla_x \tf
\|^2_2} } \right)   \right)^{\frac12} \frac{\tm(u)} {\sqrt u} \,du \,, \\ 
 & \leq C c_T \| \tf \|_{L^2_\tm(H^2)} \|\nabla f\|_{2,\frac\tm u}
\|\nabla \tf\|_{2,\tm} \left( 1+ \int \frac{\tm(u)\|\nabla \tf\|^2_2
}{\|\nabla \tf\|^2_{2,\tm} } \ln \left( \frac{ c_s\|\nabla^2_{x,x} \tf
\|^2_2}{\|\nabla_x \tf \|^2_2} \right) \,du \right)^{\frac12} \,,\\
 & \leq  C c_T \|  \tf \|_{L^2_\tm(H^2)} \|\nabla
f\|_{2,\frac \tm u} \|\nabla \tf\|_{2,\tm} \left( 1+ \ln \left(
{\textstyle \frac{
c_s \|\nabla^2_{x,x} \tf \|^2_{2,\tm}}{\|\nabla_x \tf
\|^2_{2,\tm}} }\right)\right)^{\frac12}
\end{split}
\end{equation*}
Remark that the fraction inside the logarithm is always greater than $1$ so that the square root is well defined. 

In order to get a bound on $\| \nabla_x \tf\|_{2,\frac \tm u}$, we use
that
\begin{eqnarray}
\| \nabla_x \tf \|^2_{2,\frac1 u}& = & 2 \int h(u) \frac{d\au}{|u|} =  2 \int
h(u) \diver _{ \au}\left( \frac{\au}{|u|} \right) d\au \nonumber \\
& = & - \int \nabla_\au(|\nabla_x \tf|^2)  \cdot 
\frac{\au}{|u|} \,dxd\au \leq 2 \|\nabla^2_\au \tf \|_2 \|\nabla_x \tf \|_2  \,, \label{eq:trick}
\end{eqnarray}
so that
\[
\| \nabla_x \tf\|^2_{2,\frac \tm u} =  \| \nabla_x \tf\|^2_{2,u} +
\| \nabla_x \tf\|^2_{2,\frac1 u}  \leq  2 \left( \|\nabla_x \tf \|_{2,\tm}
+ \|\nabla_\au \nabla_x \tf \|_{2,\tm} \right) \|\nabla_x f \|_{2,\tm}
\]
Therefore,
\begin{equation} \label{eq:gevol3} 
\begin{split}
\frac12   \partial_t & \|\nabla_x \tf \|_{2,\tm}   + \nu \|
\nabla_{x,\au} \nabla_x \tf    \|_{2,\tm} ^2  \leq  (4 \nu + 2\beta)  \|\nabla_x \tf
\|_{2,u^0}^2  + \\
& C c_T \| \nabla^2_{x,x} \tf \|_{2,\tm}  \|\nabla \tf\|_{2,\tm}^{\frac32} \left(
\|\nabla_x \tf \|_{2,\tm}^{\frac12} + \|\nabla_\au \nabla_x \tf
\|_{2,\tm}^{\frac12} \right)
\left( 1+ \ln \left( \frac{ c_s \|\nabla^2_{x,x} \tf
\|^2_{2,\tm}}{\|\nabla_x \tf \|^2_{2,\tm}} \right)\right)^{\frac12} \,.
\end{split} \end{equation}
We next use the inequality $\; 1 + \ln(x) \leq \frac{x^\ep}\ep$, valid for  any $\ep \in (0,1), \; x >0$, and get
\begin{equation*}
\begin{split}
   \partial_t  & \|\nabla_x \tf \|_{2,\tm}   + \nu \|
\nabla_{x,\au} \nabla_x \tf    \|_{2,\tm} ^2  \leq  (4 \nu +2\beta) \|\nabla_x \tf
\|_{2,u^0}^2  +  \ldots \\
& C(\ep) \| \nabla_{x,\au} \nabla_x \tf \|_{2,\tm}^{1+\ep}  \|\nabla_x
\tf\|_{2,\tm}^{\frac32 -\ep} \left( \|\nabla_x \tf \|_{2,\tm}^{\frac12} +
\|\nabla_{x,\au} \nabla_x \tf \|_{2,\tm}^{\frac12}\right) \,.
\end{split} \end{equation*}
with $C(\ep) = \frac{C c_T}{\ep} \sqrt{c_s}$. Now, with the temporary notations $a = \| \nabla_x \tf\|_{2,\tm}$ and $b = \| \nabla_{x,\au} \nabla_x \tf    \|_{2,\tm}$, what we need is to eliminate
all the $b$ in the right hand side, with the help of the $b$ of the left hand
side. Precisely, in the right hand side, we have the two terms
\[
b^{1+\ep} a^{2-\ep} \,, \quad \text{and } b^{\frac32 + \ep} a^{\frac32 -\ep} 
\,.
\]
We will use the Young inequalities $xy \leq \frac{x^p}p + \frac{y^q}q$, where
$\frac1p + \frac1q=1$. For the first term, with $p = \frac2{1+\ep}$ and then $q
=\frac2{1-\ep}$, and for the second with $p=\frac4{3 + 2\ep}$ and then
$q=\frac4{1-2\ep}$. We obtain the following bounds
\[
 b^{1+\ep} a^{2-\ep} \leq \frac{1+\ep}2 b^2 + \frac{1-\ep}2
a^{4+\frac{2\ep}{1-\ep}}\,, \quad \text{and } 
b^{\frac32 + \ep} a^{\frac32-\ep} \leq \frac{3+2\ep}4 b^2 + \frac{1-2\ep}4
a^{6+\frac{8\ep}{1-2\ep}} \,,
\]
valid for  $\ep <\frac12$. Taking into account the two constants $\nu$ and
$C(\ep)$ we get for $\ep
<\frac12$
\begin{equation*}
\begin{split}
\frac12   \partial_t \|\nabla_x \tf  & \|_{2,\tm}^2   + \frac \nu 2 \|
\nabla_{x,\au} \nabla_x \tf    \|_{2,\tm} ^2  \leq  (4 \nu + 2\beta) \|\nabla_x \tf
\|_{2,m}^2  +  \ldots \\
& C(\ep)^{\frac2{1-\ep}} \nu^{- \frac{1+\ep}{1-\ep}} 
\| \nabla_x \tf\|_{2,\tm}^{4+\frac{2\ep}{1-\ep}} +
   C(\ep)^{\frac4{1-2 \ep}} 
\nu^{-\frac{3+2\ep}{1-2\ep}} \| \nabla_x
\tf\|_{2,\tm}^{6+\frac{8\ep}{1-2\ep}}  \,.
\end{split} 
\end{equation*}
where $C(\ep)$ has only change from a numerical constant. With the notation $h =
\|\nabla_x \tf
\|_{2,\tm}^2$, it gives
\[
\frac12   \partial_t h \leq  (4 \nu + 2\beta)  h + C(\ep)^{\frac2{1-\ep}} \nu^{-
\frac{1+\ep}{1-\ep}} 
h^{2+\frac{\ep}{1-\ep}} + C(\ep)^{\frac4{1-2 \ep}} 
\nu^{-\frac{3+2\ep}{1-2\ep}} h^{3+\frac{4\ep}{1-2\ep}}  \,.
\]
That is a differential inequality with a growth faster than linear, that give a
solution that may
explode in a finite time. The time of explosion $\tau^*$ may be bounded below by
something depending only on $\nu, \, \ep, \, T$ and $h(0) = \| \nabla_x \tf_i
\|_{2,m}$.  Since $\ep$ may be choosen arbitrairy between $0$ and $\frac12$, $\tau^*$ depends only on $\nu$,$\beta$, $T$ and the initial value $h_0$.

Using that in the inequality \eqref{eq:gevol3}, and turning back to the original $u$ variable, we obtain the existence of a constant $C_3(T, \nu,\|\nabla_x
f_i\|_m)$ such that
\[
 \|\nabla_x f_i\|^2_{2,m} + \frac \nu 2  \int_0^{\tau^*} \|
(\nabla_x,\partial_u) \nabla_x f \|^2_{2,m} \,dt\leq C_3(T,\nu,\|\nabla_x f_i\|_{2,m})  
\]

\medskip
{ \bf A bound by below for $\tau^*$.}
In order to simplify the next two paragraphs, we assume that $\beta \leq \nu$. 
A careful analysis of the term in the right hand side of the previous equation
shows that if $\ep$ is choosen small enough so that $\nu \leq C(\ep)$, then for
$h \leq \bar h := \left( \frac \nu {C(\ep)}\right)^2$ the
dominant term is $4 \nu h$ and for $h \geq \bar h $ the
dominant term is $h^{6+\frac{8\ep}{1-2\ep}}$. 

Then, if $h(0) \geq \bar h$, the explosion time is given by the equation
\[
 \partial_t h \leq  C(\ep)^{\frac2{1-2 \ep}} 
\nu^{-\frac{3+2\ep}{1-2\ep}} h^{3+\frac{8\ep}{1-2\ep}} \,.
\]
For that equation, we get that the explosion time is larger than 
\[
 \tau^* =  C(\ep)^{- \frac4{1-2 \ep}} 
\nu^{\frac{3+2\ep}{1-2\ep}} h(0)^{- \frac{2}{1-2\ep}}
\]

In the case $h(0) \leq \bar h$, then for the early time, the equation may be
rewritten
\[
 \partial_t h \leq 12 \nu h \,,
\]
 till $h(0)  = \bar h$. It take a time greater than $T^1 = \frac1{12\nu}
\ln\left(\frac{\bar h}{h(0)}\right)$. And after that time, the explosion time
is given by the later calculation, and due to simplification it comes
$\frac C \nu$. Finally, we get an explosion time
\[
 \tau^* = \frac1{12\nu} \ln\left(\frac{\bar h}{h(0)}\right) + \frac C \nu 
\]

\medskip
{\bf Best choice for $\ep$.}
It is quite difficult to optimize that quantity in $\ep$. But, as the condition on $\ep$ are $0 < \ep < \frac12$ and $\nu \leq C(\ep) = C \frac{c_T} \ep$, we can choose 
\[
\ep  = \min \left( \frac18 , C \frac{ c_T}{\nu} \right)
\]
With that choice, we get
\begin{itemize}
\item[i)] If $\nu \leq 8 C c_T$, $\displaystyle \tau^* = \begin{cases} 
C \nu^{\frac{13}3} h_0^{-\frac83} \quad \text{if } h_0 \geq C \left( \frac \nu {c_T}\right)^2 \\
\frac C \nu \left( 1 + \ln\left( {\textstyle \frac{C \nu^2}{c_T^2 h_0 }}\right) \right)  \quad \text{ else}
\end{cases}$
\item[ii)] If $\nu \geq 8 C c_T$, $\displaystyle \tau^* = \begin{cases} 
C \nu^{\frac{13}3} h_0^{-\frac83} \quad \text{if } h_0 \geq 1 \\
\frac C \nu ( 1 - \ln h_0 ) \quad \text{ else}
\end{cases}$
\end{itemize}
It is then clear  that the value of $\tau^*$ depends only on $\nu,h_0$ and the temperature $T$. 

The case of physical interest is the first one. Since in core of tokamaks, we have a large temperature $T$ which implies a constant $c_T$ large, and a small colisionnality, in other words a small $\nu$. 
 
 \cqfd

\sectionnew{Existence of solutions.} \label{sect:existence}
In this section, we prove the existence theorem \ref{thm:existence}.
The proof will use  the following notation and a preliminary lemma. 
%
A priori estimates of the Lemma \ref{lem:u-moment} on the solution $(f,\Phi )$ to \ref{eq:gyroFP2D}-\ref{eq:elecneutr} on $[0,T]$ lead to the definition of the set $K$ of functions $f$ such that
\[
\|f(t) \|_{2,m}  \leq \sqrt M, \quad a.a. t\in 0,T,
\]
where
\[
M= \| f_i \|^2_{2,m} + (2\nu +\beta)  \frac{e^{2 \beta t} -1 } \beta \| f_i \|^2_{2,2\pi u} 
\]
For each $n >0$, we also introduce an approximation of the potential $\Phi_n$ defined for any $f \in L^2_m$ by
\begin{eqnarray}  \label{eq:phiapprox}
\Phi _n (t,x):= \sum _{|k|\leq n; k\neq 0}e^{ik\cdot x} \frac{1} {1-\hat{H}(k)}\Big( \int 2 \pi J^0_w\hat{f_n}(t,k,w)wdw-1\Big) ,
\end{eqnarray}

\begin{lemma} \label{lem:approx} 
For any $n \in \N^*$ and any $T>0$, there is a unique $f_n$ in $K\cap  L^2(0,T; H^1_u(\Omega ))$ solution to \eqref{eq:gyroFP2D} with the potential $\Phi$ replaced by  $\Phi _n= \Phi _n(f_n)$  
and initial condition $f_i$. That solution satisfy all the a priori estimate of the previous section.  
\end{lemma} 
\underline{Proof of Lemma \ref{lem:approx}} Let $S$ be the map defined on $K$ by $S(f)= g$, where $g$ is the solution in $K\cap  L^2(0,T; H^1_u(\Omega ))$ to \eqref{eq:gyroFP2D} with the potential $\Phi _n(f)$ and initial condition $f�$. The existence and uniqueness of $S(F)$ follows from \cite{Lions-Magenes} Thm 4.1 p 257, since $\nabla \Phi _n$ is bounded in $L^{\infty }(0,T; H^3(\T ^2 ))$ by $c_n M$ for some constant $c_n$. Then $S$ maps $K$ into $K$. Moreover,  $S$ is a contraction in $L^{\infty }(0,T;L^2_u(\Omega ))$ for $T$ small enough. Indeed, let $g_1= S(f_1)$ (resp. $g_2= S(f_2)$). By estimates very similar to the one performed in Lemma \ref{lem:rhoreg} it holds that 
\begin{eqnarray*} 
\forall  t\geq 0, \quad  \| (\Phi _n(f_1)- \Phi _n(f_2))(t, \cdot) \|_{L^{\infty }(\T ^2)}\leq  \bar{c_n}\| (f_1- f_2)(t,\cdot ) \|_{2,m},
\end{eqnarray*} 
for some constant $ \bar{c_n}$. Substracting  the equation satisfied by $g_2$ from the equation satisfied by $g_1$ and integrating over $\Omega $ leads to
\begin{equation*} \begin{split}
\frac{e^{2(2\nu + \beta)t} }{2} & \frac{d}{dt}  \left( e^{- 2(2\nu + \beta)t}\| g_1 - g_2\|_{2,m}^2 \right)   \\
&   \leq - \nu \| (\nabla_x,\partial_u)(g_1 - g_2)\|_{2,m}^2 - \int g_2 \nabla^\perp (J_u^0(\Phi_n(f_2)-\Phi_n(f_1))) \cdot \nabla (g_1-g_2)  \,m(u) \,dxdu \\
& \leq - \nu \| (\nabla_x,\partial_u)(g_1 - g_2)\|_{2,m}^2 +  \bar{c_n}\| (f_1 - f_2)\|_{2,m}\| \nabla _x(g_1 - g_2)\|_{2,m}\|g_2\|_{2,m} \\
& \leq \frac{\bar{c_n}^2}{4 \nu^2} \| (f_1 - f_2)\|_{2,m}^2\|g_2\|_{2,m}^2 \leq \frac{\bar{c_n}^2 M }{4 \nu^2} \| (f_1 - f_2)\|_{2,m}^2
\end{split} \end{equation*}
And so
\begin{eqnarray*}
\| g_1- g_2 \|_{L^{\infty }(0,T; L^2_m) } \leq c Te^{2(\nu +\beta )T}\| f_1- f_2 \|_{L^{\infty }(0,T; L^2_m) }.
\end{eqnarray*}
Hence there is a unique fixed point of the map $S$ on $[0,T_1]$ for $T_1$ small enough. The bounds used for defining $T_1$ being independent of $T_1$, a unique solution of the problem can be determined globally in time by iteration. The fact that this unique solution satisfy the a-priori estimates of the next section is clear since these estimates only depends on the bound satisfied by $\Phi$ and not on its precise form. 
\cqfd

\underline{Proof of Theorem \ref{thm:existence}}

The sequence $(f_n)$ is compact in $L^2_{loc,u}((0,T)\times \Omega )$. Indeed, it is bounded in \\
$L^{\infty }(0,T;L^2_u(\Omega_U))\cap L^2(0,T;H^1_u(\Omega_U))$ for any bounded subset $U >0$ (Recall of $\Omega_U = \T^2 \times \R^+$). It follows from the interpolation theory that  $(f _n)$ is bounded in $L^{\frac{10}{3}}_u((0,T)\times \T ^2\times U)$. Together with the boundedness of $(\nabla _ x\Phi _n(f _n))$ in $L^2_{loc,u}((0,T)\times \Omega )$, this implies that $(\frac{\partial f_n}{\partial t})$ is bounded in $W^{-1,\frac{5}{4}}_{loc,u}((0,T)\times \Omega_U )$. By the Aubin lemma \cite{Lions-Magenes}, it holds that $(f _n)$ is compact in $L^2_u((0,T)\times \Omega_U)$, so converges up to a subsequence to some function $f$ in $L^2_u$. \\
It remains to pass to the limit when $n\rightarrow +\infty $ in the weak formulation satisfied by $f _n$. A weak form of \eqref{eq:gyroFP2D}-\eqref{eq:indata}  is that for every smooth test function $\alpha $ with compact support in $[0,T[\times \Omega $,
\begin{multline} \label{eq:weakform}
\int f_i(x,u)\alpha (0,x,u) \,u\,dxdu+\int _0^t \int f _n\Big( \frac{\partial \alpha }{\partial t}+\nabla _ x^{\perp }(J_u^0\Phi _n(f_n))\cdot \nabla _x \alpha \Big) \,u\,dxduds \\
= \int _0^t \int \Big( u \nu \nabla _x f_n\cdot \nabla _x \alpha + \partial_u f_n \partial_u \alpha + \beta u^2  f_n\partial_u \alpha \Big) \,dxduds \,.  
\end{multline}
The passage to the limit in  \eqref{eq:weakform} when $n\rightarrow +\infty $ can be performed if
\begin{eqnarray*}
\lim _{n\rightarrow \infty }\int _0^t \int uf _n\nabla _ x^{\perp }(J_u^0(\Phi _n(f_n))\cdot \nabla _x \alpha dxduds= \int _0^t \int uf\nabla _ x^{\perp }(J_u^0(\Phi (f))\cdot \nabla _x \alpha dxduds.
\end{eqnarray*}
This holds since $(f_n)$ (resp. $(\nabla _ x (J_u^0(\Phi _n)(f _n)))$ strongly (resp. weakly) converges to $f$ (resp. $\nabla _ x (J_u^0(\Phi (f))$) in $L^2_{loc,u}((0,T)\times \Omega )$. And since the $f_n$ satisfy all the a priori bound, the limit $f$ also satisfies them. 
\cqfd

\sectionnew{Short time uniqueness and stability of the solution.} 
\label{sect:uniqueness}

In this section we prove the shot time uniqueness and stability theorem \ref{thm:uniqueness}
\medskip

\underline{Proof of Theorem \ref{thm:uniqueness}.} Denote by $f_1$ (resp. $f_2$) a solution to \ref{eq:gyroFP2D} for the field $\Phi_1$ (resp. $\Phi_2$), by $\delta f= f_1-f_2$ and by $\delta \Phi = J_u^0(\Phi _1-\Phi _2)$. Multiplying the equation satisfied by $(1+u^2) \delta f$ by $\delta f$ and integrating w.r.t. $(x,u)$ with the weight $u$ leads to
\begin{eqnarray*}
\frac{1}{2} \frac{d}{dt} \| \delta f\|_{2,m}^2  & \leq  & 
- \nu \| (\nabla_x,\partial_u)\delta f\|_{2,m}^2 + (4 \nu +2 \beta) \|\delta f \|_{2,u}^2+\int \delta f\nabla_x^{\perp }(\delta \Phi)\cdot \nabla_x f_2 \,m(u)\,dxdu  \\
& \leq & - \nu \| (\nabla_x,\partial_u)\delta f\|_{2,m}^2 + 4 \nu \|\delta f
\|_{2, m}^2 + \|\delta f  \nabla_x(\delta \Phi)\|_{2, um} \|\nabla_x f_2\|_{2,\frac m u }.
\end{eqnarray*}
To estimate $\|\delta f \nabla(\delta \Phi)\|_{2,u\,m }$, apply the
inequality
\[
 ab \leq e^a -b + b \ln b \,, \qquad  a,b >0,
\]
to $(a,b)= \left( \left(\frac{\nabla_x \delta \Phi}{ 6\|\nabla_x^2 \delta
\Phi\|_2}\right)^2 ,\left( \frac{\delta f}{\|\delta f\|_2} \right)^2 \right) $
for every nonnegative $u$ and apply the Tr\"udinger inequality  (See
\cite{Moser})
\begin{equation} \label{eq:exp-int}
\int_{\T^2} e^{\left( \frac{\nabla_x \delta \Phi}{6\|\nabla_x^2\psi
\|_2}\right)^2 } \,dz \leq 2 \,.
\end{equation}
Therefore, using $\|\nabla_x^2 \delta \Phi \|_2 \leq \frac{C c_T}{\sqrt u} \| \nabla
\delta f\|_{2,m}$ (Lemmas \ref{lem:rhoreg} and \ref{lem:phireg} and \ref{lem:boundL}) and the Jensen
inequality
\begin{eqnarray*}
 \|\delta f \nabla_x (\delta \Phi) \|_2^2 & = & C \|\nabla^2 \delta \Phi\|_2^2
\|\delta f\|_2^2  \int  \left(\frac{\delta f}{\|\delta f\|_2}
\right)^2
\left(\frac{ |\nabla_x \delta \Phi| }{  6 \|\nabla^2 \delta \Phi\|_2}\right)^2 
\,dx \\
& \leq & C \|\nabla^2 \delta \Phi\|_2^2 
\|\delta f\|_2^2   \left( 1+ \int \frac{ (\delta f)^2 }{\|\delta f\|_2}
\ln\left(\frac{(\delta f)^2}{\|\delta f \|_2^2}\right)\,dx \right). \\
& \leq & \frac{C c_T^2 }{u} \| \nabla_x \delta f\|_{2,m}^2
\|\delta f\|_2^2   \left( 1+ \ln \left( \frac{\|\delta f\|^4}{\|\delta f \|_2^4}
\right) \right)
\end{eqnarray*}
Integrating in $u$ with the weight $um$, it holds using again Jensen
inequality that
\begin{eqnarray*}
   \| \delta f \nabla(\delta \Phi) \|^2_{2,um} 
& \leq & 2C c_T^2 \| \nabla_x \delta f\|_{2,m}^2 \| \delta f\|_{2,m}^2
\int   \frac{\|\delta f\|_2^2 }{\| \delta f\|_{2,m}^2} \left( 1+ \ln \left(
\frac{\|\delta f\|^2_4}{\|\delta f \|_2^2}
\right)\right)  m(u)\,du \\
& \leq & C c_T^2 \| \nabla_x \delta f\|_{2,m}^2 \| \delta f\|_{2,m}^2
 \left( 1+ \ln \left( \frac{\|\delta f\|^2_{L^2_m(L^4)}}{\|\delta f \|_{2,m}^2}
\right)\right) \,. \\
\end{eqnarray*}
Consequently,
\[ \begin{split}
\frac{1}{2} \frac{d}{dt} \| \delta f\|_{2,m}^2   \leq  C c_T  \|\nabla_x \delta
f\|_{2,m}  & \| \delta f \|_{2,m} \sqrt{1+ \ln \left( \frac{\|\delta
f\|^2_{L^2_m(L^4)}}{\|\delta f \|_{2,m}^2}
\right)} \, \|\nabla f_2\|_{2,\frac m u } \\
   & - \nu \| (\nabla_x,\partial_u)\delta f\|_{2,m}^2 +  (4 \nu+ 2 \beta)  \|\delta
f \|_{2,m}^2 \,,
\end{split} \]
and finally using the inequality $\|f(t)\|_{L^2_m(L^4)} \leq e^{(\beta-2\nu)t} \|f_i\|_{L^2_m(L^4)}$ from lemma \ref{lem:disp} 
\begin{eqnarray*}
\frac{1}{2} \frac{d}{dt} \| \delta f\|_{2,m}^2   & \leq  & \frac{C c_T^2}{4\nu} 
  \| \delta f \|^2_{2,m} \ln\left(\frac{2 e \|\delta
f\|^2_{L^2_m(L^4)}}{\|\delta f  \|_{2,m}} \right) \|\nabla
f_2\|^2_{2, \frac m u}  + (4 \nu + 2\beta) \| \delta f\|_{2,m}^2 \,, \\
& \leq & \frac{C c_T^2}{4\nu} 
  \| \delta f \|^2_{2,m} \ln \left( \frac{r^2 e^{2(\beta + 2 \nu) t+ 1}}{\|\delta f  \|^2_{2,m}} \right) \|\nabla
f_2\|^2_{2, \frac m u}  + (4 \nu + 2\beta) \| \delta f\|_{2,m}^2  \,,
\end{eqnarray*}
where
\begin{eqnarray*}
r = e^{\frac12} \,\left(\| f_{1,i} \|_{L^2_m(L^4)} + \| f_{2,i} \|_{L^2_m(L^4)} \right)  \,.
\end{eqnarray*}
Defining $s(t) =  \frac 1 {r^2} \| \delta f\|_{2,m}^2  e^{-2(\beta + 2 \nu) t}$, we get
\[
\dot s (t) \leq \frac{C c_T^2}{4\nu}  \|\nabla f_2\|^2_{2, \frac m u} \, s(t) \ln \frac 1 {s(t)} \,.
\]
It follows from the Osgood lemma that
\begin{equation}\label{eq:osg}
s(t) \leq s(0)^{e^{-H(t)}}
\end{equation}
with $H(t) = \frac{C c_T^2}{4\nu}  \int_0^t \|\nabla f_2(s)\|^2_{2, \frac m u} \,ds$. We will show below that $H$ is well defined on $[0,\tau^*]$, the time defined in Lemma \ref{lem:grad_xf}. It implies
\[
\| \delta f (t) \|_{2,m} \leq e^{(\beta+2\nu)t}\, r^{1-e^{-H(t)}} \, \| \delta f (t) \|_{2,m}^{e^{-H(t)}} \,
\]
and from that inequality we get the short time uniqueness and stability. Remark that the previous calculation do not use $\nabla_x f_1$ and this is why we do not need an assumption on this quantity in the stability result.  

\medskip
It remains to prove that $H$ is bounded on $[0,\tau^*]$. In fact,
\[
\int_0^{\tau^*} \|\nabla f_2\|^2 _{2,(1+u^2)^{\frac72}} \,dt \leq 2^{\frac52} \int_0^{\tau^*} \|\nabla f_2\|^2 _{2,(1+u^7)} \,dt ,
\]
since $(1+u^2)^{\frac72} \leq 2^{\frac52} (1+ u^7)$. Moreover, $\int_0^{\tau^*} \|\nabla f_2\|^2 _{2,u^0} \,dt $ may be bounded using Lemma \ref{lem:grad_xf} and the inequality \eqref{eq:trick} (we emphasize that the form used here is slighty different because we are in here in the $1D$ setting for the variable $u$)
\begin{eqnarray*}
\int_0^{\tau^*} \|\nabla f(t)\|^2 _{2,u^0} \,dt & \leq & 2 \int_0^{\tau^*}  \|\nabla^2_{x,\au} f(t) \|_{2,u} \|\nabla_x f(t) \|_{2,u}  \\
& \leq &  \sup_{t \leq \tau^*} \|\nabla_x f(t) \|_{2,u}  \sqrt{\tau^*} \left( \int_0^{\tau^*} \|\nabla^2_{x,\au} f(t) \|_{2,u} \,dt \right)^{\frac12} \\
& \leq & C^* \sqrt{\frac{\tau^*} \nu }  \,.
\end{eqnarray*}
And $\int_0^{\tau^*} \|\nabla f_2\|^2 _{2,u^7} \,dt $ is bounded by   Lemma \ref{lem:u-moment}, with $n=3$. \cqfd

\appendix

\section{Appendix:  Some controls on the Bessel function of zero-th order.}
\label{App:A}

The first Bessel function $J^0$ is much used in this paper. Indeed, in Fourier space, $J^0_1$ that appears in the definition of the gyroaverage of the electric field, is the multiplication by $J^0$. Some properties of the function $J^0$ are given in \cite{Watson}. In this appendix, some bounds on $J^0$ and its derivative are proven.
\begin{lemma} \label{lem:boundJ}
$J^0$ satisfies the following estimates for all $k \in \R$
\begin{eqnarray*}
& i) &\qquad |J^0(k)|  \leq  \min \left( 1, \frac{1}{2^{1/4} \sqrt{ k}} \right) ,\\
& ii) &\qquad |J^0(k)|   \leq  (1+k^2)^{^{-\frac14}} \,, \\
& iii) &\qquad |(J^0)'(k)|  \leq \min \left( 1, \sqrt{\frac{2}{\pi k }} \right) \,, \\
& iv) & \qquad |(J^0)'(k)| \leq (1+k^2)^{^{-\frac14}} \,.
\end{eqnarray*}
\end{lemma}

\underline{Proof of Lemma \ref{lem:boundJ}}

{\it First Inequality :}
The bound $|J^0(k)| \leq 1$ is clear from the definition of $J^0$,
\begin{equation} \label{eq:J0bis}
J^0(k) = \frac{1}{2\pi}\int_0^{2\pi} e^{ik\cos \theta} \,d\theta = \frac{1}{\pi} \int_0^\pi \cos(k\cos \theta)\,d\theta \,.
\end{equation}
The bound by $(\sqrt 2 k)^{-\frac12}$ is obtained as follows. $J^0$ is solution of the ordinary differential equation
\begin{equation}
k^2 (J^0)'' + k (J^0)' + k^2 J^0 = 0 \,, \quad J^0(0)=1 \,, \; (J^0)'(0) = 0 \,.
\end{equation}
The new unknown $u = \sqrt{k}J^0$ is solution to
\[
u'' + \left( 1 + \frac{1}{4k^2} \right) u = 0 \,. \qquad 
\]
There are no exact initial conditions for $u$. However,
\[
u(k) \underset{ \scriptscriptstyle k \rightarrow 0^+}{=} \sqrt k [1+ O(k^2)] \,, \quad u'(k) \underset{\scriptscriptstyle k \rightarrow 0^+}{=} \frac1{2\sqrt k} [1+ O(k^2)] .
\]

The second equation admits the $k$-dependent energy,
\[
H(k) = H(k,u,u')= \frac{u'^2}{2} + \frac{u^2}{2}\left( 1+ \frac{1}{4k^2} \right) \,,
\]
that satisfies
\[
H(k) - H(k_0) = - \int_{k_0}^k \frac{u^2(l)}{4l^3} \,dl .
\]
It follows from the behaviour of $u$ near $0$ that
\[
H(k)  \underset{ \scriptscriptstyle k \rightarrow 0^+}{=} \frac1{4k} + O(k) \,.
\]
Moreover the series expansion of $J^0$ near $k=0$,
\[
J^0(k) = \sum_{j=0}^\infty (-1)^{j} \frac{k^{2j}}{2^{2j}(j!)^2}
\]
and its alternate character if $k \leq 2$ imply that $u^2(k) \geq k - \frac{k^3}{2}$ (valid for $k \leq \sqrt{2}$). Using, the inequality and the behavior of $H$ near $0$, we get if $0 < k_0 \leq k \leq \sqrt 2$
\begin{eqnarray}
H(k) & \leq & \frac{1}{4 k_0} + O(k_0)  - \int_{k_0}^k \left( \frac{1}{4l^2} - \frac{1}{8} \right) \,dl,  \nonumber\\
H(k) & \leq & \frac{1}{4k} + \frac{k}{8} \,, \label{eq:Hk}
\end{eqnarray}
since the first line is satisfied for any $k_0 >0$.
Therefore,
\begin{eqnarray*}
u^2(k) & \leq & k \frac{k^2+2}{4k^2+1}, \quad  k \leq \sqrt{2} \,.
\end{eqnarray*}
A simple calculation shows that the function appearing in the right hand side is increasing in $k$, so that 
\[
u^2(k) \leq \frac1{\sqrt 2},\quad k\in  [0,\sqrt 2]\,.
\]
For $k \geq \sqrt 2$, simply remark  that $H$ is decreasing and that from \eqref{eq:Hk}
\[
u^2(k) \leq 2 H(k) \leq  2 H(\sqrt 2) \leq \frac1{\sqrt2}
\]
In any case we get $u^2(k) \leq 2^{-\frac12}$ which gives the desired inequality.

\medskip

{ \it Second inequality : } It is a consequence of the first, for $k \geq 1$.
For $k \leq 1$, it may be obtain from a comparison of the entire development of $J^0$ and $(1+k)^{-1/4}$ around the origin. We get
\[
J^0(k) \leq 1 - \frac{k^2}{4} + \frac{k^4}{64} \leq 1 - \frac{k^2}{4} + \frac{5 k^4}{32} - \frac{15 k^6}{128} \leq (1+k^2)^{-1/4}
\]

\medskip

{ \it Third inequality :}  Taking the derivative of $J^0$ in the definition \eqref{eq:J0},
\begin{equation*}
(J^0)'(k)  =  \frac{i}{2\pi} \int _0^{2\pi} \cos \theta e^{i k \cos \theta} \,d\theta = - \frac{1}{\pi} \int _0^\pi \cos \theta \sin(k \cos \theta) \,d\theta\,  ,
\end{equation*}
from which it is clear that $| (J^0)'(k)| \leq 1$ for all $k$. Next we transform the previous integral in
\begin{eqnarray*} 
(J^0)'(k) & = & -\frac{2}{\pi} \int_0^1 \frac{\alpha \sin(k\alpha)}{\sqrt{1-\alpha^2}}\,d\alpha \, ,\\
& = &  \sum_{i=0}^{j-1} (-1)^j \int_{h_i}^{h_{i+1}} \frac{|\sin (k\alpha)|}{\sqrt{1-\alpha^2}} \alpha \,d\alpha :=  \sum_{i=0}^{j-1} (-1)^j s_j\,,
\end{eqnarray*}
where $(h_i)_{1\leq i\leq j}$ are the points where $\sin(k\theta)$ vanishes and $1$,
\[
h_0=0 < h_1=\frac{\pi}{k} < h_2 = \frac{2\pi}{k} < \ldots < h_{j-1} = \frac{(j-1)\pi}{k} < h_j = 1.
\] 
The previous sum has alterned sign, the larger terms occuring for large $i$. Its terms are with increasing  absolute values, except for the last one which is incomplete and may be smaller than the next to last term. However,
\[
 -s_1 \leq s_0 -s_1 \leq \sum_{i=0}^{j} (-1)^j s_j \leq s_0 -s_1 + s_2 \leq s_0,
\]
so that
\begin{eqnarray*}
|(J^0)'(k)| & \leq & \max(s_0,s_1) \leq \frac{2}{\pi} \int_{1-\pi/k}^1 \frac{\alpha d\alpha}{\sqrt{1-\alpha^2}} \\
& \leq & \frac{1}{\pi}  \sqrt{\frac{2\pi}{k} - \frac{\pi^2}{k^2}} \leq \sqrt{\frac{2}{\pi  k}},\quad k \geq \pi .
\end{eqnarray*}
This ends the proof of the third inequality.\\
The proof of $iv)$ is similar to the proof of $ii)$, since $\sqrt{\frac2\pi} < 2^{-\frac14}$. 
\cqfd

\bibliographystyle{plain}
{\bibliography{Gyro_FP}} 

\def\cprime{$'$} \def\cprime{$'$}
\begin{thebibliography}{10}

\bibitem{bouchut}
Fran{\c{c}}ois Bouchut.
\newblock Smoothing effect for the non-linear
  {V}lasov-{P}oisson-{F}okker-{P}lanck system.
\newblock {\em J. Differential Equations}, 122(2):225--238, 1995.

\bibitem{Brizard04}
A.~J. {Brizard}.
\newblock A guiding-center fokker-planck collision operator for nonuniform
  magnetic fields.
\newblock {\em Physics of Plasmas}, 11:4429--4438, September 2004.

\bibitem{FreSon}
Emmanuel Fr{\'e}nod and Eric Sonnendr{\"u}cker.
\newblock The finite {L}armor radius approximation.
\newblock {\em SIAM J. Math. Anal.}, 32(6):1227--1247 (electronic), 2001.

\bibitem{GheHauNou09}
Philippe Ghendrih, Maxime Hauray, and Anne Nouri.
\newblock Derivation of a gyrokinetic model. {E}xistence and uniqueness of
  specific stationary solution.
\newblock {\em Kinet. Relat. Models}, 2(4):707--725, 2009.

\bibitem{Gysela06}
V.~{Grandgirard}, Y.~{Sarazin}, X.~{Garbet}, G.~{Dif-Pradalier}, P.~{Ghendrih},
  N.~{Crouseilles}, G.~{Latu}, E.~{Sonnendr{\"u}cker}, N.~{Besse}, and
  P.~{Bertrand}.
\newblock {GYSELA, a full-f global gyrokinetic Semi-Lagrangian code for ITG
  turbulence simulations}.
\newblock In O.~{Sauter}, editor, {\em Theory of Fusion Plasmas}, volume 871 of
  {\em American Institute of Physics Conference Series}, pages 100--111,
  November 2006.

\bibitem{Ladyzen}
O.~A. Lady{\v{z}}enskaja, V.~A. Solonnikov, and N.~N. Ural{\cprime}ceva.
\newblock {\em Linear and quasilinear equations of parabolic type}.
\newblock Translated from the Russian by S. Smith. Translations of Mathematical
  Monographs, Vol. 23. American Mathematical Society, Providence, R.I., 1967.

\bibitem{Landau}
L.D. Landau.
\newblock The transport equation in the case of coulom interactions.
\newblock In D.~ter Haar, editor, {\em Collected papers of L. D. Landau}, pages
  163--170. Pergamon Press, Oxford, 1981.

\bibitem{Lions-Magenes}
J.-L. Lions and E.~Magenes.
\newblock {\em Probl\`emes aux limites non homog\`enes et applications. {V}ol.
  1}.
\newblock Travaux et Recherches Math\'ematiques, No. 17. Dunod, Paris, 1968.

\bibitem{Moser}
J{\"u}rgen Moser.
\newblock {A sharp form of an inequality by Trudinger.}
\newblock {\em Indiana Univ. Math. J.}, 20:1077--1092, 1971.

\bibitem{Watson}
G.~N. Watson.
\newblock {\em A treatise on the theory of {B}essel functions}.
\newblock Cambridge Mathematical Library. Cambridge University Press,
  Cambridge, 1995.
\newblock Reprint of the second (1944) edition.

\end{thebibliography}

\end{document}